\documentclass[reqno]{article}
\usepackage{amssymb,amsmath,amsthm}

\usepackage[T1]{fontenc}
\usepackage[latin1]{inputenc}
\usepackage{graphicx}
\usepackage{enumerate,ae}

\newtheorem{Theo}{Theorem}
\newtheorem{Cor}[Theo]{Corollary}
\newtheorem{Lem}[Theo]{Lemma}

\theoremstyle{remark}
\newtheorem{remm}{Remark}

\def\E{\mathbb{E}}
\def\Z{\mathbb{Z}}

\def\CL{\mathcal{L}}

\def\epsilon{\varepsilon}

\newcommand{\indz}[1]{\,\mathbb{I}_{#1}}

\newcommand{\norm}[1]{\left\Vert #1\right\Vert}
\newcommand{\normdeux}[1]{\Vert #1\Vert_{2}}
\newcommand{\normsup}[1]{\Vert #1\Vert_{\infty}}

\def\tops{\xrightarrow[]{a.s.}}

\def\toloi{\xrightarrow[]{\CL}}

\def\var{\mathrm{Var}}
\def\cov{\mathrm{Cov}}

\def\argmin{\ \mathrm{Argmin}}

\def\E{\mathbb{E}}
\def\Z{\mathbb{Z}}

\def\CL{\mathcal{L}}

\def\epsilon{\varepsilon}

\def\var{\mathrm{Var}}
\def\cov{\mathrm{Cov}}
\def\argmin{\ \mathrm{Argmin}}

\begin{document}

\title{Estimating and forecasting partially linear models with non stationary exogeneous variables}

\author {Xavier Brossat$^{\text\small{1}}$  \, Georges
  Oppenheim$^{\text\small{2}}$\, and \, Marie-Claude Viano$^{\text\small{3}}$ \\
\
{\small $^{^{\text\small{1}}}$
Département OSIRIS -- Service de Recherche et Développement -- EDF}  \\
{\small 1, av. du G\'en\'eral de Gaulle 92140 Clamart, France}\\
{\small 
$^{^{\text\small{2}}}$ 
Laboratoire d'Analyse et de Mathématiques Appliquées Université Paris-Est}\\
{\small 5, bd. Descartes,
Cité Descartes - Champs-sur-Marne}\\
{\small 77454 Marne-la-Vallée cedex 2, France}\\
{\small $^{^{\text\small{3}}}$ Laboratoire Paul Painlev\'e  UMR CNRS 8524 -- Bat
  M2.} \\
{\small Universit\'e Lille 1, Villeneuve d'Ascq, 59655 Cedex, France.}}
\date{}
\maketitle
\begin{abstract}
\noindent This paper presents a backfitting-type method for estimating and forecasting a periodically
correlated partially linear
model with exogeneous variables and heteroskedastic input noise. A rate of
convergence of the estimator is given. The results are valid even if the
period is unknown. \\
\begin{center}
\textbf{R\'esum\'e}
\end{center}
On utilise une procédure itérative de type backfitting pour estimer les
paramètres d'une classe de  modèles partiellement linéaires périodiquement
corrélés présentés pour modéliser l'évolution de la consommation
d'électricité. On obtient une vitesse de convergence des estimateurs et un
intervalle de prévision de consommation.

\end{abstract}

\vskip .5cm 
  
{\bf keywords:}   {$\alpha$-mixing, additive models, backfitting, electricity consumption, forecasting interval,
semiparametric regression smoothing.}


\section{Introduction}
In this paper, we focus on partially  linear models of the type
\begin{equation}\label{gen}
X_{n}= \sum_{j=1}^pa_j X_{n-j}+\sum_{j=0}^qb_j(e_{n-j})+\sigma(e_n,\ldots,e_{n-q'})\varepsilon_n.
\end{equation}
The parameters $p\geq 1$ and $q\geq 0$ are supposed known while the coefficients $a_j$ as well as the functions $b_j$ and $\sigma$ are
unknown. The sequence $(\varepsilon_n)$ is an unobserved system noise. The aim is to predict $X_{n+h}$, for some $h\geq 1$, from $((X_n,e_n),(X_{n-1},e_{n-1}),\ldots)$, the observed set
of past values available at date $n$.

During the last 20 years, partially linear autoregressive models such as (\ref{gen}) have gained
attention, as being a good compromise between linear models and purely non parametric
ones. Such models, proposed in \cite{Eng} to represent the relationship
between weather and electricity consumption are now widely used in the
literature.  See for example \cite{Har} where a chapter is devoted to models
including (\ref{gen}). The functions $b_j$ are expanded on a suitable basis and
the first coefficients of this expansion, together with the $a_j$'s,  are estimated
via a L.M.S method. See also \cite{Har2}. With the same type of partially linear models \cite {Fad,Gan} use wavelets in the estimation scheme. In \cite {Bat}, the $b_j$'s are
treated as nuisance parameters. Let us also mention \cite{Gao, Gao1, Gao2}, devoted to models including purely autoregressive ones, where
some past values operate in a linear form and the others in a functional one. These authors use an orthogonal series method, and propose a data based criterion to determine the truncation parameters. See also chapter 8 in \cite{Fan2}, where models like
$$
X_{n+1}=f(X_n)+ aX_{n-1}+\sigma(X_n)\varepsilon_n. 
$$
include linear and non linear autoregressive summands together with some volatility. The functional parts are estimated via local linear estimators and gaussian limits for the renormalized errors are obtained. 

Model (\ref{gen}) presents several
advantages. Firstly, the additive form reduces the so-called curse of dimensionality. Secondly,
linear autoregression is preserved when expressing the future values $(X_{n+h})_{h=1,\ldots}$ from the past ones $(X_{n-h},e_{n-h})_{h=0,\ldots}$, which makes it easier, and in some sense coherent, forecast at lags greater than $1$.
Lastly, model (\ref{gen}) is
specially well adapted to the situation where the output $X_n$ is electricity consumption at date $n$
and the input $e_n$ the temperature at the same date, since it is well-known that the effect
of temperature on electricity sales is highly non-linear at
extreme temperatures, while linearity of the autoregression seems to be a reasonable assumption. Notice that, in practical situations, the temperature at date $n$ is either measured or forecasted by Météo-France. In both cases, the value of the exogeneous variable $e_n$ is known. 
Accurate electrical load forecasting is essential for power utilities. Electricit\'e de France (EDF) performs a climatic correction. The influence of temperature on electric demand is widely reported. Other extra so-called exogenous variables are included in the short-term models. They may be random variables like wind speed, or deterministic ones like "position-within-the year of the date" which is a year-periodic variable. With those variables, for horizons up to 3 days by 1/2 hourly steps, the forecasts are very efficient when based on nested models studied during many years.

For simplicity and convenience, we only consider in this paper the situation $q=q'=0$, leading to the model
\begin{equation}\label{lina}
X_{n}= a_1X_{n-1}\ldots+a_pX_{n-p}+b(e_n)+\sigma(e_n)\varepsilon_n,\qquad n\in \Z.
\end{equation}
The algorithm presented below can easily be adapted to the general case $q,q'>0$, and the results of theorem \ref{resum0} still hold with a loss of speed if $b$ or $\sigma$ have non-additive forms.

\subsection{Elements of discussion}
\subsubsection{Backfitting} 
Backfitting methods, first proposed  by \cite{Buj},  
are usually recommended for additive models which involve several explanatory
variables, each having an unknown functional form. The method is well described in
\cite {Fan2,Hast}. See also  \cite{Fan1,Ops, Ops1} where the estimation algorithms
use local polynomial regression and \cite{Mam} based on
projections on polynomial spaces. The performances of backfitting procedures when
autoregression is involved are less well understood. In 
\cite{Wan}, for the non linear stationary autoregressive model with exogeneous variables
$$
X_n=a(X_{n-1})+b(e_n)+\varepsilon_n,
$$
the algorithm works in two steps: the first step builds a preliminary estimator of $a$ et $b$ by piecewise constant functions. Then, from the obtained pseudo remainders, the second step builds kernel estimators of the same functions. 
The author obtains the limit law for the estimation error.

For the model (\ref{gen}),
if the period $T$ is known, a simple estimation scheme would consist of splitting the data in $T$ subsamples, each of them being a trajectory of a stationary process. Then the parameter $\theta={^t(a_1,\ldots,a_p)}$ and the function $b()$ could be estimated separately, the first one at the usual parametric rate, and the second one at the slower usual functional rate (see \cite {Fan2,Spec} for remarks on this question). 
The choice of a backfitting scheme for estimating (\ref{gen}) presents the advantage of allowing the period of the input sequence $(e_n)$ to remain unknown. As it will be proved below, the price to pay for this is a slower rate in the estimation of $\theta$. Note that simulation studies seem to indicate that the iterative method presented below still works even when the period shows slight variations.
Within the backfitting iterations, a kernel based statistics estimates the functional part of the model. Other methods could have been used here (local estimators, splines, wavelets for instances). Usually tuning the bandwidth, through a cross-validation process, enhances the estimators' quality. We haven't studied that point for two reasons: no theoretical results are available and we wished to study the bare quality of the basic estimators. The underlying questions are postponed to a future paper.
\subsubsection{Parameters $p$ and $q$} It could be interesting to estimate the orders $p$ and $q$ of the autoregression and regression parts. In a first approach, we suppose that these parameters are known. In fact, in the particular situation of forcasting electricity consumption, these parameters have been widely studied and are supposed to be known. The order $p$ is large, but the characteristic polynomial has only few non-zero coefficients, so that the coefficients $a_j$ are to be estimated under constraints. Our convergence results can easily be extended to that sort of situation. 
\subsubsection{Comments on the results} Sections \ref{main} and  \ref{int} hereafter mainly consists of asymptotic results. These results are formulated as 
$
\hat\theta_n-\theta =O(u_n)
$
for a sequence $u_n$ going to zero. Several complements are missing: 
\begin{itemize}
\item Is the rate $u_n$ exact?
\item If that is the case, how do we get an idea of the constant in $O(u_n)$?
\item What about the limit distribution of the re-normalized error? 
\end{itemize}
The almost sure (a.s) convergence  is the only type of studied convergence. No central limit theorem is included. The usual developments of expectation and variance, even when they are included, are not brought forward. 
The a.s. convergence is all that is needed to compute the forecast interval as far as the asymptotic interval is concerned. The innovation distribution quantiles are all that we need.
For the last question, a complementary study is in progress, in order to obtain gaussian limits as it is the case in this kind of studies (see for example \cite{Fan2}). The section here devoted to simulations attempts to answer the first questions. See for example, Figure \ref{fig:speed2} and the comments in section \ref{rate}.  

\section{Estimation of the parametric and non parametric components}
The aim is to estimate the functions $b(.)$ and $\sigma(.)$ and the  vector parameter
$$
\theta={^t(a_1,\ldots,a_p)}.
$$
Denoting
$$
\phi_k={^t(X_{k-1},\ldots,X_{k-p})},
$$
the model can be written
\begin{equation}\label{linavec}
X_n={^t\phi_n}\theta+b(e_{n})+\sigma(e_n)\varepsilon_{n}.
\end{equation}

We choose a kernel $K$, and a smoothing parameter $h_n$. 

Having chosen initialised estimation of $\theta$ and a stopping rule, the iterative method consists of estimating $\theta$ (resp.  $b$) by using an estimation of the residual calculated from the previous estimation of $b$ (resp. $\theta$).   
\begin{itemize}
\item Initialisation. Fix the first value $\hat \theta^{(1)}$
\item Step  1. Estimate the function $b$ by a kernel estimator based on the partial residuals 
$$
\hat b_n^{(1)}(e)=\frac{\sum_{l=p+1}^{n}\left(X_{l}-{^t\phi_l}\hat \theta^{(1)}\right)K_n\left(e-e_l\right)}{\sum_{l=p+1}^{n-1}K_n\left(e-e_l\right)}
$$
where 
$$
K_n(e):=K\left(\frac{e}{h_n}\right).
$$
\item Step 2. Update the estimation of $\theta$ by a least mean squares estimator based on the new partial residuals 
\begin{eqnarray*}
\hat \theta_n^{(2)}&=&\argmin_{\theta} \sum_{l=p+1}^{n}(X_{l}-^t\phi_l\theta-\hat b_n^{(1)}(e_l))^2\\
&=&\Sigma_n^{-1}\sum_{l=p+1}^{n-1}\phi_l(X_{l}-\hat b_n^{(1)}(e_l))
\end{eqnarray*}
with 
\begin{equation}\label{matcov}
\Sigma_n=\sum_{l=p+1}^{n}\phi_l ^t\phi_l.
\end{equation}

Finally, the transition from step $k-1$ to step $k$ can be expressed as
\begin{eqnarray}\label{backeq}
\hat b_n^{(k-1)}(e)&=&\frac{\sum_{l=p+1}^{n}\left(X_{l}-^t\phi_l\hat \theta^{(k-1)}\right)K_n\left(e-e_l\right)}{\sum_{l=p+1}^{n-1}K_n\left(e-e_l\right)}\\
\hat \theta_n^{(k)}&=&\Sigma_n^{-1}\sum_{l=p+1}^{n}\phi_l(X_{l}-\hat b_n^{(k-1)}(e_l)).
\end{eqnarray}

\item  Chosing a stopping time $k$ for the iterations, the variance $\sigma^2(e)$ is then estimated by a kernel method using the partial residuals based on the estimates $\hat\theta^{(k)}_n$ and $\hat b_n^{(k-1)}$
\begin{equation}\label{sigma}
{\hat\sigma^2}_{n,k}(e)=\frac{\sum_{l=p+1}^{n-1}\left(X_l-^t\phi_l\hat\theta^{(k)}_n-\hat b_n^{(k-1)}(e_{l})\right)^2K_n\left(e-e_l\right)}{\sum_{l=p+1}^{n-1}K_n\left(e-e_l\right)}
\end{equation}
\end{itemize}
As in the case of linear regression, estimating $\theta$ and $b$ does not need any estimation of $\sigma$, implying that $\hat\sigma_{n,k}$ is obtained at the end of the iterative scheme. See \cite{Fan2} for remarks on this so-called oracle effect. 
\section{Main results}\label{main}
\subsection{Hypotheses}\label{Hypcomp}

We adopt the following basic hypotheses ($\mathcal H$). 
\begin{itemize}
\item $\mathcal H_1$: \underline{Periodicity}. The exogeneous sequence $(e_n)$ is the sum of a periodic deterministic sequence $(s_n)$ and a bounded zero-mean strong white noise
\begin{equation}\label{per}
e_n=s_n+\eta_n \quad\forall n
\end{equation}
\item $\mathcal H_2$: \underline{Whiteness of the system noise}. $(\varepsilon_n)$ is a bounded i.i.d sequence of zero-mean variables, and $\var(\varepsilon_n)=1$.
\item $\mathcal H_3$: \underline{Stability}. The autoregressive dynamic is stable. In other words, the polynomial
$$
A(z)=z^p-\left(\sum_{j=1}^pa_jz^{p-j}\right)
$$
does not vanish on the domain $\vert z\vert\geq 1$.
\item $\mathcal H_4$: \underline{Independence of the inputs}. The two sequences $(\varepsilon_n)$ and $(\eta_n)$ are independent.
\item $\mathcal H_5$: \underline{On the distributions of input sequences}. The distributions of  $\varepsilon_1$ and $\eta_1$ both have a density. 
The density $f$ of $\eta_1$ is continuous and non-vanishing on the support
$[-m_{\eta},m_{\eta}]$ of $\eta_1$. The density $g$ of $\varepsilon_1$ is $C_1$
and never vanishes on the  support $[-m_{\varepsilon},m_{\varepsilon}]$ of $\varepsilon_1$.
\item $\mathcal H_6$: \underline{On the functions}. 
Let $\mathcal E=\cup_{j=1}^T[s_j-m_\eta,s_j+m_\eta]$ denote the union of the $T$ compact supports of the variables $e_j$.\\
\begin{enumerate}
\item The function $b$ is $\gamma$-Hölderian on $\mathcal E$, for some $0<\gamma\leq 1$, which means that
\begin{equation}\label{Hol}
\sup_{e_1,e_2\in \mathcal E}\frac{\vert b(e_1)-b(e_2)\vert}{\vert e_1-e_2\vert^\gamma}<\infty
\end{equation}
\item The variance $\sigma^2(e)$ of the input noise is $\gamma_1$-Hölderian on $\mathcal E$, for some $0<\gamma_1\leq 1$, and 
\begin{equation}\label{varnoise}
\inf_{e\in\mathcal E}\sigma(e)>0.
\end{equation}
\end{enumerate}
\item $\mathcal H_8$: \underline{On the kernel}.
The kernel $K$ is lipschitzian, and satisfies
$$
\int K(u) du=1
$$
Keeping in mind the example of electricity consumption, hypothesis $\mathcal H_1$ allows some periodicity in the random structure of the input sequence $(e_n)$. Boundedness of the noises (hypotheses $\mathcal H_1$ and $\mathcal H_2$) is assumed only to have shorter proofs. Without this boundedness, the uniform speed in Theorem \ref{resum0} only holds on compact sets. 
Hypothese $\mathcal H_5$ assures that the denominators of the kernel-type estimators of $b$ and $\sigma$ are not asymptotically vanishing. The Hölder exponents $\gamma$ and $\gamma_1$ in hypothesis $\mathcal H_6$, govern the convergence rate of the estimation scheme. 

In what follows, we work with the periodically correlated solution of (\ref{linavec}) defined in section \ref{prel}.
\end{itemize}

\subsection{Existence of $\hat \theta_n^{(k)}$}
The lemma below establishes that, almost surely, the matrix $\Sigma_n=\sum_{l=p+1}^{n}\phi_l ^t\phi_l$ appearing in (\ref{matcov}) and used in estimating  the parameter $\theta$ is invertible at least for large enough $n$. 
\begin{Lem}\label{cov}
Under the hypotheses $\mathcal H_{1,2,3,4} $, the matrix $\Sigma_n$ being defined in  (\ref{matcov}), as $n\to \infty$,

(i)
$$
\frac{\Sigma_n}{n}\tops M=\frac{1}{T}\sum_{l=0}^{T-1}\E\left(\phi_0^{(l)}{^t\phi_0^{(l)}}\right)=\frac{1}{T}\sum_{l=0}^{T-1}\left[\mu^{(l)}{^t\mu^{(l)}}+\Gamma^{(l)}\right]
$$ 
where
\begin{equation}\label{fil}
\phi_k^{(l)}={^t(X_{kT+l},X_{kT+l-1},\ldots, X_{kT+l-(p-1)})}.
\end{equation}
and where 
$\mu^{(l)}=\E(\phi_0^{(l)})$ and $\Gamma^{(l)}$ is the covariance matrix of $\phi_0^{(l)}$.

(ii) The limit matrix $M$ is regular.
\end{Lem}
The proof is in the Appendix.

\subsection{Analysis of estimation errors}
We first focus on the estimation errors
\begin{equation}
\tilde \theta_n^{(k)}=\theta-\hat\theta_n^{(k)} \quad\hbox{and}\quad
\tilde b_{n}^{(k-1)}(e)=b(e)-\hat b_n^{(k-1)}(e).
\end{equation}
From (\ref{linavec}),
\begin{eqnarray}
\tilde \theta_n^{(k)}&=&-\Sigma_n^{-1}\sum_{l=p+1}^{n}\phi_l\left(\tilde b_{n}^{(k-1)}(e_l)+\sigma(e_l)\varepsilon_l\right)\label{erreq}\\
\tilde b_{n}^{(k-1)}(e)&=&\frac{\sum_{l=p+1}^{n}\left(-^t\phi_l\tilde \theta_n^{(k-1)}+b(e)-b(e_l)-\sigma(e_l)\varepsilon_l\right)K_n\left(e-e_l\right)}{\sum_{l=p+1}^{n}K_n\left(e-e_l\right)}\nonumber\\\label{erreq2}. 
\end{eqnarray}
Thanks to the linearity of $\tilde b_{n}^{(k-1)}(e)$ with respect to $\tilde \theta_n^{(k-1)}$, this leads to the linear recusrive equation

\begin{equation}\label{errtheta}
\tilde \theta_n^{(k)}=A_n\tilde\theta_n ^{(k-1)}+R_n^{(1)}+R_n^{(2)}
\end{equation}
where 
\begin{eqnarray}
A_n&=&
\Sigma_n^{-1}\left(\sum_{l=p+1}^n\phi_l
\frac{\sum_{j=p+1}^n{^t\phi_j}K_n\left(e_l-e_j\right)}{\sum_{j=p+1}^{n}K_n\left(e_l-e_j\right)}\right)\label{An}\\
R_n^{(1)}&=&\Sigma_n^{-1}\sum_{l=p+1}^{n}\phi_l
\frac{\sum_{j=p+1}^{n}\left(b(e_j)-b(e_l)+\sigma(e_j)\varepsilon_j\right)K_n\left(e_l-e_j\right)}{\sum_{j=p+1}^{n}K_n\left(e_l-e_j\right)}\label{1}\\
R_n^{(2)}&=&\Sigma_n^{-1}\sum_{l=p+1}^{n}\phi_l\sigma(e_l)\varepsilon_l\label{3}
\end{eqnarray}

\subsection{Convergence results}
Considering (\ref{errtheta}), we are going to prove that, as $n\to \infty$, the matrix operator $A_n$ converges to a strictly shrinking one. As emphazised in \cite{Buj} this is the key result implying that $\tilde \theta_n^{(k)}$ stabilizes as $k$ increases. Then we prove that the remainder term $R_n^{(1)}+R_n^{(2)}$ tends to zero, which implies that the stabilizing value $\tilde \theta_n^{(\infty)}$ in turn vanishes when $n\to\infty$. 
This leads to the main result, whose detailed proof is in the Appendix.
\begin{Theo}\label {resum0}
With the assumptions of section \ref{Hypcomp},
if the smoothing parameter is such that, as $n\to\infty$
$h_n\sim n^{\beta_1}(\ln n)^{\beta_2}$,
there exists $\beta\in]0,1[$ such that
\begin{equation}\label{result1}
\left.\begin{array}{l}
\normdeux {\hat\theta_n^{(k)}-\theta}\\
\sup_{e\in \mathcal E}\vert\hat b_n^{(k)}(e)-b(e)\vert\\
\end{array}\right\}
=O_{a.s.} \left(\sqrt \frac{\ln
 n}{nh_n}\right)+O_{a.s.}(h_n^\gamma)+O_{a.s.}(\beta^{k})
 \end{equation}
 and
 \begin{equation}\label{variance}
 \sup_{e\in \mathcal E}\vert {\hat\sigma^2}_{n,k}(e)-\sigma^2(e)\vert=O_{a.s.} \left(\sqrt \frac{\ln
 n}{nh_n}\right)+O_{a.s.}(h_n^{\min\{\gamma,\gamma'\}})+O_{a.s.}(\beta^{k})
 \end{equation}
where the $0(.)$'s are  uniform with respect to $k$ and $n$.
\end{Theo}
We see that the convergence rate of $\hat\sigma^2_{n,k}(e)$ cannot exceed that of the other parameters and can even be slower when $b(e)$ is smoother than $\sigma(e)$. The equality (\ref{variance}) is proved in the Appendix.

As a result, an optimal rate is obtained by chosing convenient values for $\beta_1$
and $\beta_2$.

\begin{Cor}\label{resum}
Under the same hypotheses, if $h_n\sim \left(\ln n/n\right)^{\frac{1}{2\gamma+1}}$, there exists $\beta\in]0,1[$ such that
\begin{displaymath}
\left.\begin{array}{l}
\normdeux {\hat\theta_n^{(k)}-\theta}\\
\sup_{e\in \mathcal E}\vert\hat b_n^{(k)}(e)-b(e)\vert\\
\end{array}\right\}
=O_{a.s.}\left(\frac{\ln n}{n}\right)^{\frac{\gamma}{2\gamma+1}}+O_{a.s.}(\beta^{k})
 \end{displaymath}
 and
 \begin{equation*}
 \sup_{e\in \mathcal E}\vert {\hat\sigma^2}_{n,k}(e)-\sigma^2(e)\vert=O_{a.s.}\left(\frac{\ln n}{n}\right)^{\frac{\min\{\gamma,\gamma'\}}{2\gamma+1}}+O_{a.s.}(\beta^{k})
 \end{equation*}
\end{Cor}

It is clear that,
provided $\beta$ is not too close to $1$, the convergence of
the term $\beta^k$ to zero is fast. In other words, stabilisation of the
iterations is easily obtained while  convergence of $\left(\frac{\ln
  n}{n}\right)^{\frac{\gamma}{2\gamma+1}}$ to zero requires large sample
size. More precisely, taking $k=k(n)\geq C\ln n$ gives

\begin{Cor}\label{rresum}
Under the same hypotheses as in Corollary \ref{resum} and with $h_n\sim \left(\ln n/n\right)^{\frac{1}{2\gamma+1}}$, if the recursive scheme stops after $k(n)\geq C\ln n$ iterations
\begin{displaymath}
\left.\begin{array}{l}
\normdeux {\hat\theta_n^{(k(n))}-\theta}\\
\sup_{e\in \mathcal E}\vert\hat b_n^{(k(n))}(e)-b(e)\vert\\
\end{array}\right\}
=O_{a.s.}\left(\frac{\ln n}{n}\right)^{\frac{\gamma}{2\gamma+1}}
\end{displaymath}
and
\begin{equation*}
\sup_{e\in \mathcal E}\vert {\hat\sigma^2}_{n,k(n)}(e)-\sigma^2(e)\vert=O_{a.s.}\left(\frac{\ln n}{n}\right)^{\frac{\min\{\gamma,\gamma'\}}{2\gamma+1}}
\end{equation*}
\end{Cor}

\begin{remm}\label{pôle}
With the above remark in mind, it is interesting to note that, when the autoregression is close to the
instability domain, the value of $\beta$ can approach  $1$. In such
situations, a large number of iterations is needed before the stabilisation of
the iterative scheme.  For example consider the particular
model 
$$
X_n=aX_{n-1}+b(e_n)+\varepsilon_n
$$
where the
sequence $(e_n)$ is i.i.d.
From Lemma  \ref{AAPE},
$$
A=\frac{\E(X_n)^2}{\E(X_n)^2+\sigma(0)}=\frac{1}{1+\frac{\sigma(0)}{\E(X_n)^2}},
$$
where  $\sigma(0)=c^2/(1-a^2)$ and $\E(X_n)=c'/(1-a)$. Hence,
$$
A=\frac{1}{1+C\frac{1-a}{1+a}}\to 1\quad\hbox{if}\quad a\to 1.
$$
Consequently, the iterative scheme can be very slow if  $a$ is close to $1$. On the
opposite,  when $a$ is close to $-1$, the iterations stabilize very
quickly.
\end{remm}

\subsection{Improvement of the rate for smooth functions $b$}
As well-known in functional estimation, a smoother $b$ induces, with some
extra conditions on the kernel $K$, a better rate of convergence of the
estimators. 
\begin{Cor}\label{better}
If the function $b$ is $C_{\ell}$ for some integer $\ell> 1$ and if the kernel satisfies
\begin{eqnarray}
\int e^kK(e) de &=&0\quad\forall k\in{1,\ldots,\ell}\label{kerneg}\\
\int K(e) de &=&1
\end{eqnarray}
(i) if the smoothing parameter is such that, as $n\to\infty$,
$h_n\sim n^{\beta_1}(\ln n)^{\beta_2}$
there exists $\beta\in]0,1[$ such that
\begin{eqnarray*}
\normdeux {\hat\theta_n^{(k)}-\theta}&=&O_{a.s.} \left(\sqrt \frac{\ln n}{nh_n}\right)+O(h_n^\ell)+O_{a.s.}(\beta^{k})\\
\sup_{e\in \mathcal E}\vert\hat b_n^{(k)}(e)-b(e)\vert&=&O_{a.s.} \left(\sqrt \frac{\ln n}{nh_n}\right)+O(h_n^\ell)+O_{a.s.}(\beta^{k})
\end{eqnarray*}
(ii) if $h_n\sim \left(\ln n/n\right)^{\frac{1}{2\ell+1}}$ the rate of the two
first terms is optimal and becomes 
$$
O_{a.s.}\left(\frac{\ln n}{n}\right)^{\frac{\ell}{2\ell+1}},
$$
\end{Cor}
The proof, based on the fact that, using (\ref{kerneg}),
$
\int (b(vh_n+e)-b(e))K(v)f(vh_n+e)dv=O(h_n^{\ell}),
$
is omitted.
\section{Forecasting intervals}\label{int}
The natural predictor for $X_{n+1}$, 
$$
\E(X_{n+1}\vert e_{n+1},e_n,\ldots,e_1, X_n,\ldots,X_1)={^t\phi_{n+1}}\theta+b(e_{n+1})
$$
can be evaluated via the estimates of $\theta$ and $b$ based on the observations
up to time $n$. In other words, we propose the predictor 
$$
\hat X_{n+1}={^t\phi_{n+1}}\hat \theta_n+\hat b_n(e_{n+1}).
$$
It should be clear that, under the conditions of Corollary \ref{resum}, 
$$
\frac{\hat X_{n+1}- X_{n+1}}{\hat\sigma_n(e_{n+1})}\toloi \varepsilon_1,
$$
and, consequently, building a prediction interval requires an estimation of
the noise's quantile function $Q(t)$. The inverse of $Q$ can be
consistently estimated by
$$
\hat Q_n^{-1}(a)=\frac{1}{n}\sum_{j=1}^{n-1}\indz{\left\vert\frac{\hat X_{j+1,n}- X_{j+1}}{\hat\sigma_n(e_{j+1})}\right\vert>a},
$$
based on the set of
retroactive predictions $
\hat X_{j+1,n}={^t\phi_{j+1}}\hat \theta_n+\hat b_n(e_{j+1}),\quad j\leq n-1.
$
which use the estimates available at time $n$.

Summarizing, for $X_{n+1}$ we obtain the prediction interval at asymptotic level $\alpha$ 
$$
\left [\hat X_{n+1}-\hat \sigma_{n}(e_{n+1})\hat Q_n(\alpha)~,~\hat X_{n+1}+\hat \sigma_{n}(e_{n+1})\hat Q_n(\alpha)\right]
$$

\section{Simulation examples: results and comments}
Several aspects of the present paper, but not all, rely on the EDF modeling-forecasting process. Some are theoretical and some others practical. Let us mention some of them which are of interest when performing the data simulation part.\\
$\bullet$ Question 1. What is the influence of a single irregular point of the  function $b$ on the regular points estimators? \\
$\bullet$ Question 2. The results are proved when both $k\to \infty$ and $n\to \infty$. Could we use a simplified procedure base on one iteration $k=1$ of the inner loop? Do we actually get benefit from the iteration process? Remember we do not use any preliminary estimator. From a practical point of view, can we get any information linking the stochastic process dependencies to a good value of $k$?\\
$\bullet$ Question 3. When the EDF engineers  estimate a model, the question of the sample size $n$ is a recurrent one. A sample could be said to be large when either its cost is high or when the distance of the statistic distribution (computed with $n$ observations) to its limit distribution is small. Theorem \ref{resum0} and Corollaries \ref{rresum} and \ref{better} offer a speed of convergence where the constants, as often, are missing.\\
The simulations answer some of these questions. 
 
Three type of autoregressions are chosen, two of order one and one of order 4. 
\begin{enumerate}
\item An AR1 process with a positive coefficient 
\begin{equation}\label{+}
X_{n+1}=0.7 X_n+b(e_n)+\sigma(e_n)\varepsilon_{n+1}
\end{equation}
\item An AR1 process with a negative coefficient 
\begin{equation}\label{-}
X_{n+1}=-0.7 X_n+b(e_n)+\sigma(e_n)\varepsilon_{n+1}
\end{equation}
\item An AR4 process
\begin{equation}\label{4}
X_{n+1}=a_1X_n+a_2X_{n-1}+a_3X_{n-2}+a_4X_{n-3}+b(e_n)+\sigma(e_n)\varepsilon_{n+1}
\end{equation}
where the roots of the characteristic polynomial are
$\pm 0.5$ and $0.5\pm 0.25$.
\end{enumerate}
The functions $b$ and $\sigma$ are the same in all examples
$$
b(e)=\sqrt{\vert e\vert},\qquad \sigma(e)=1+\frac{e^2}{24}.
$$
The input white noise $\varepsilon_n$ has a standard gaussian distribution, and the exogeneous $e_n=s_n+\eta_n$ where $s_n$ is a 6-periodic sequence with $s_0=-1.2,\;s_1=3.1, \;s_2=1.80,\;s_3= -2.51,;\ s_4=-3.2,\;s_5= -0.25$ and where the noise $\eta_n$ is i.i.d. with marginal distribution uniform on $[-3,+3]$.
\subsection{Three examples}

For each of the three models (\ref{+}), (\ref{-}) and (\ref{4}), a trajectory of size $n=5000$ is simulated and the estimations of the parameter $\theta$ and of the functional parameter $b$ are carried over through a number $k$ of iterations varying from $1$ to $50$. Having reached the last iteration, the estimation of $\sigma^2$ is then calculated.
The kernel $K$ is the gaussian kernel and the smoothing parameters $h_n$ and $h'_n$, used in the estimations of $b$ and $\sigma$, are
\begin{equation}\label{hn}
h_n=1.5 \hat S_e n^{-1/2}\quad\hbox{and}\quad h'_n= 0.15 \hat S_e n^{-1/3}
\end{equation}
where $\hat S_e$ is the empirical standard deviation of the $e_j$'s.

The results are depicted in Figures \ref{fig:estimationFig2}, \ref{fig:estimationFig3} and  \ref{fig:estimationFig4}. The upper-left graphic shows the function $b(e)=\sqrt{\vert e\vert}$, its estimate after $50$ iterations together with the cloud of partial residuals used to calculate the estimate (see formula (\ref{backeq})). The upper-right graphic shows $b$ and the evolution of its estimations $\hat b_n^{(k)}$ as $k$ varies from $1$ to $50$. The lower-left graphic shows the evolution of the estimator of the AR parameter $\theta$ as a function of the number $k$ of iterations, and the lower-right one presents the standard deviation $\sigma(e)$ and its final estimation. 
\subsubsection{Model (\ref{+})}
Four main effects are noticeable. 
\begin{itemize}
\item As the number of iterations increases, the estimates of $b$ and of $\theta$ improve. 
\item  The iterations stabilize very slowly. This is not surprising since the value of the parameter $\theta=0.7$ is close to $1$ (see Remark \ref{pôle} just after Corollary \ref{rresum}).
\item As expected, for fixed $k$, the convergence of $\hat b_n^{(k)}(e)$ is far worse in the neighbourhood of $e=0$, discontinuity point of $b'$. This effect is even still visible for the estimator of $\sigma(e)$ (lower-right graphic), despite the smoothness of this function at this point. 
\end{itemize}
\subsubsection{Model (\ref{-})}
Compared with the first example, there are only two differences
\begin{itemize}
\item The iterations stabilize quickly ($4$ iterations are enough), due to the fact that $\theta=-0.7$ is close to $-1$, 
\item But the obtained limit value of $\hat\theta_n^{(4)}$ is not very close to the true value, meaning that in this case, more observations are needed for a good estimation. However, the estimate of $\sigma(e)$ seems quite good.
\end{itemize}
\subsubsection{Model (\ref{4})}
In this example $\theta$ has $4$ components. They are indicated, in the lower-left graphic, by $4$ horizontal lines. The stabilisation point of the iterations is between those obtained in the two other examples ($40$ iterations are enough), perhaps due to the presence of the positive root $0.5$. The sample size is large enough to get good estimations. It seems that the order of the autoregression, at least for moderate orders, has no significant effect on the quality of the method. 

\subsection{Evolution of the estimation errors as functions of the sample size}\label{rate}
In the two last sections, we take model (\ref{-}) 
and we simulate sample paths for sizes going from $200$ to $10000$. For each sample path, the estimations of the three parameters $\theta$, $b$ and $\sigma^2$ and of the distribution of the noise are computed, based on $k=20$ iterations. Then the estimation errors are calculated.
Except for the error on $\theta$, we compute three sorts of errors, based on $L_1$, $L_2$ and $L_\infty$ norms:
\begin{itemize}
\item For the functional parameters $b$ and $\sigma$, denoting by $d$ the length of the domain of $e$, we choose

$$
N_1(h)=\frac{1}{\sqrt d}\int \vert h(e)\vert de,\qquad N_2(h)=\sqrt{\int h^2(e) de}\quad \hbox{and}\quad N_\infty(h)=\sqrt d \normsup h
$$
which satisfy $N_1\leq N_2\leq  N_{\infty}$
\item For the noise distribution, we compute the total variation, the Hellinger and the Kolmogorov distances. 
\end{itemize}
Moreover, in order to reduce fluctuations, we simulate fifty independent trajectories for each sample size, and compute the average of the errors obtained from these trajectories. 

The averaged errors are presented in Figures \ref{fig:speed1} and \ref{fig:speed2} which show, from top to bottom and left to right,  the error on $\theta$, $b$, $\sigma$, and on the noise distribution (three curves in each of the three last graphics, corresponding to different distances). The abscissa is the sample size $n$. 

Figure \ref{fig:speed1} presents clearly the fact that the convergence to zero of all the errors becomes very slow when $n$ is larger than $2000$, meaning that the asymptotic speed $(\ln n/n)^{1/4}$ is reached. Errors seem to quickly decrease for small sizes. 

Figure  \ref{fig:speed2} is a log log set of graphics. The (nearly!) straight lines represent $c(\ln n /n)^{1/4}$ for five values of $c$. Except for the error on the noise distribution, which decreases faster,
the theoretical bound $n^{-1/4}$ (see Corollary \ref{rresum} with $\gamma= 1/2$) looks exact. 
\subsection{Stopping time for the iterations}
We chose to stop the backfitting iterations when the estimations are stabilized: namely, after the first $k$ such that
\begin{eqnarray*}
\max\left\{\normdeux{\hat\theta_n^{(k)}-\hat\theta_n^{(k-1)}},
N_1(b_n^{(k)}-\hat b_n^{(k-1)})\right\}\leq 10^{-3}
\end{eqnarray*}
Let us denote by $k(n)$ the obtained stopping point. As pointed out in Corollary \ref{rresum}, $k(n)$ should be of order $\ln n$, hence hardly varying in the domain $n\leq 1000$. 

For each model, and each sample size $n$, five independent trajectories are simulated.
This is illustrated in Figure \ref{fig:choixk1}, for the three models (\ref{+}), (\ref{-}), (\ref{4}). The sample size varies between $100$ and $1000$. There are three groups of 5 piecewise linear lines. Model (\ref{-}) is represented by the lines in the lower part of the graphic. For this model, the stopping point is almost constantly equal to $7$ and $8$. Models (\ref{4}) (darkest lines) and (\ref{+}) occupy the upper part. This illustrates the asymptotic theory and the observations in Figures \ref{fig:estimationFig2}, \ref{fig:estimationFig3} and  \ref{fig:estimationFig4}. 


\begin{figure} [p]
	\centering
	\includegraphics[scale=.5]{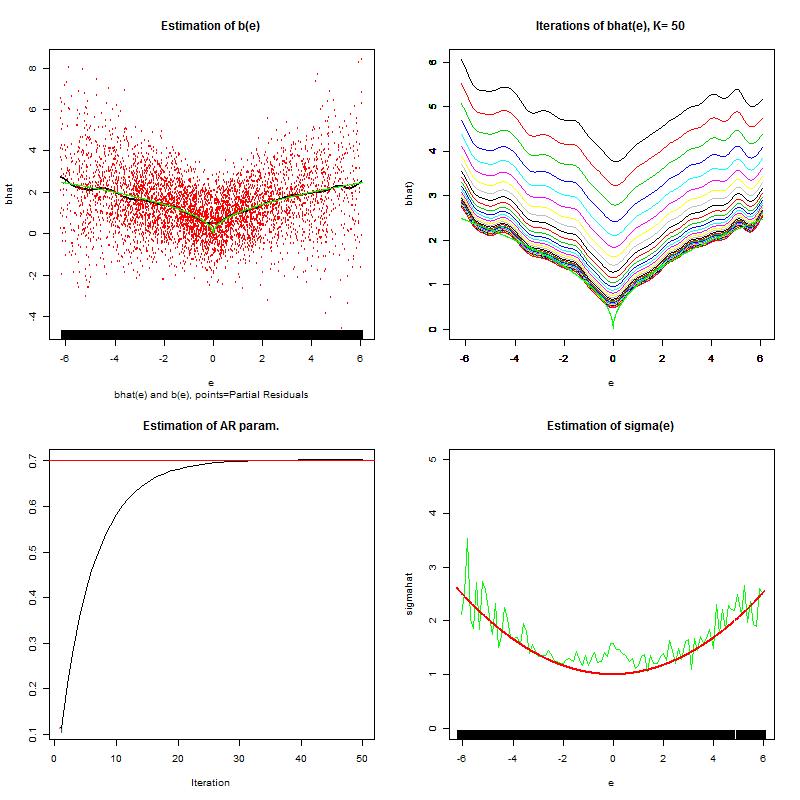}
	\caption{Estimation results for model (\ref{+}). The trajectory length is $n=5000$ and $k=50$ 
iterations are performed. The upper-right figure shows the evolution of $\hat{b}^{(k)}(e)$ 
for $k=1:50$.  The lower-left figure presents the evolution of the estimator of $\theta$ for $k=1:50$. 
The iterations stabilize very slowly, but the size of the sample is enough to obtain good estimation. The lower-right figure presents the estimator of $\sigma(e)$ for $k=50$.}
	\label{fig:estimationFig2}
\end{figure}

\begin{figure} [p]
	\centering
		 \includegraphics[scale=.5]{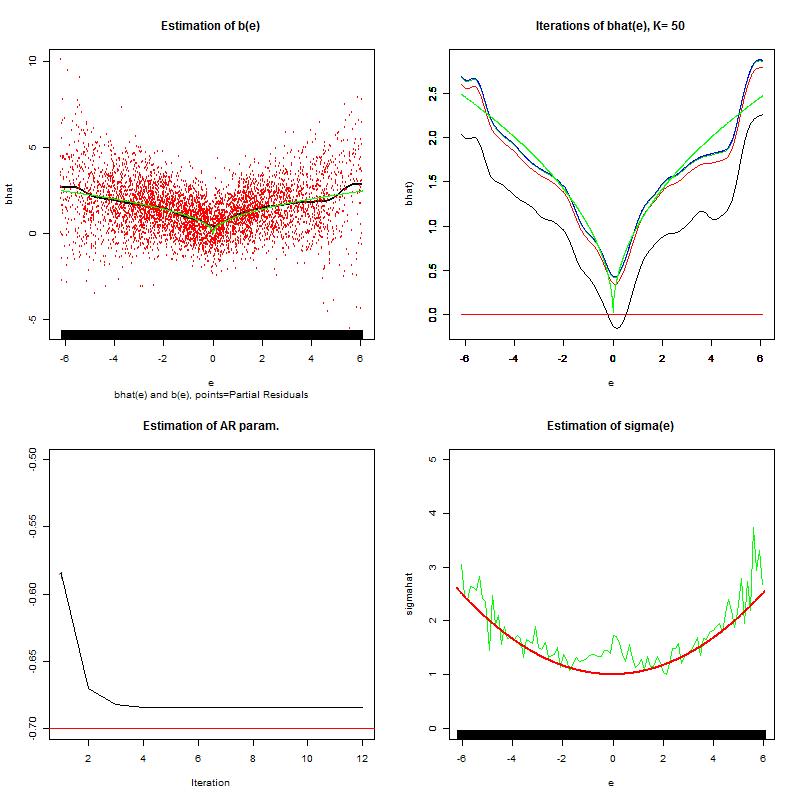}
	\caption{Estimation results for model (\ref{-}). Few iterations are needed, but the size sample seems to be too small to obtain a good estimation of $\theta$. Nevertheless, the estimation of the functions looks satisfactory.}
	\label{fig:estimationFig3}
\end{figure}

\begin{figure} [htbp]
	\centering
		 \includegraphics[scale=.5]{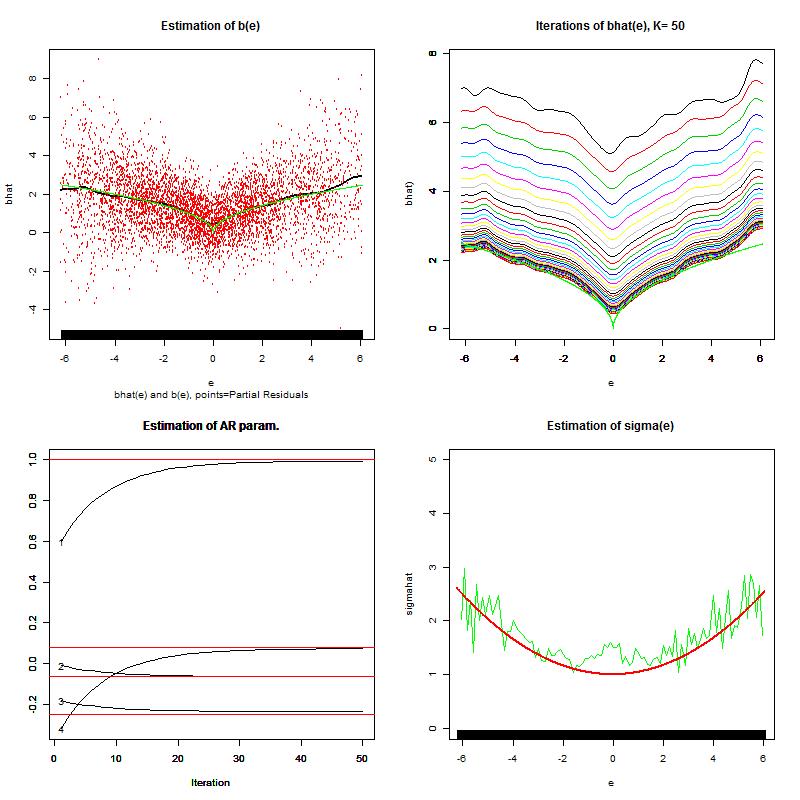}
	\caption{Estimation results for model (\ref{4}). The coordinates of $\theta$
are the horizontal lines on the lower-left figure. The iterations converge slowly ($40$ iterations), and the estimations are rather good.}
	\label{fig:estimationFig4}
\end{figure}


%

\begin{figure} [htbp]
	\centering
		 \includegraphics[scale=.5]{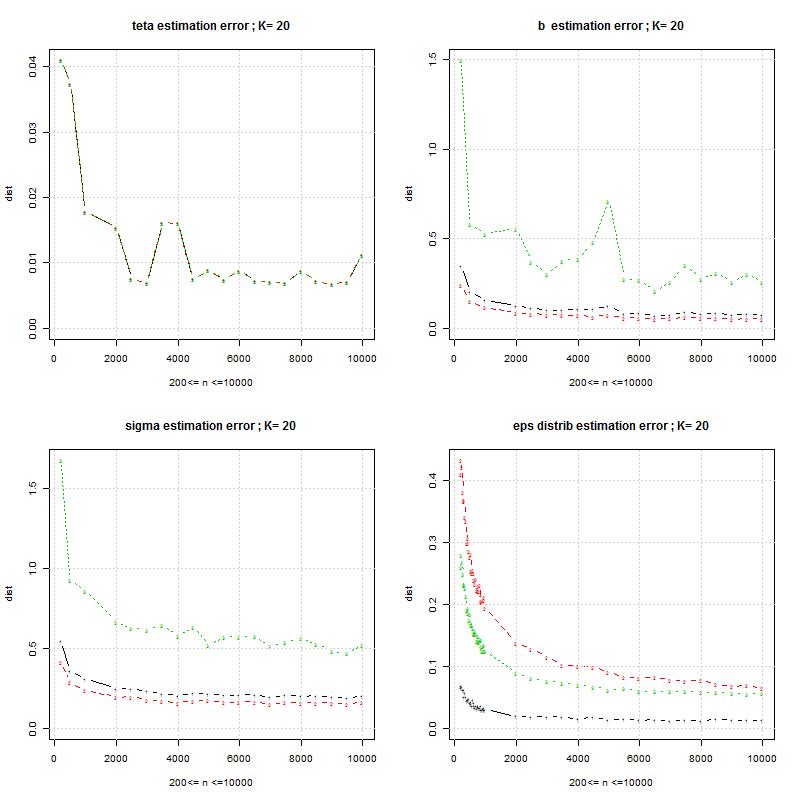}
	\caption{Estimation error as a function of the length $n$ of the trajectory. The  model is (\ref{-}). From top to bottom and left to right: errors on the estimation of $\theta$, $b$, $\sigma$, and the distance between the estimated distribution of the noise  and the Gaussian distribution. In abscissa, the two first values are $200$ and $500$. Then, the lag remains equal to $500$. In graphics 2 and 3, the positions of the three curves are conform to inequalities $N_1\leq N_2\leq N_\infty$. For the lower-right figure, the distances are (top to bottom) are Total-Variation, Hellinger and Kolmogorov-Smirnov distances. 
}
	\label{fig:speed1}
\end{figure}
\begin{figure} [htbp]
	\centering
		 \includegraphics[scale=.5]{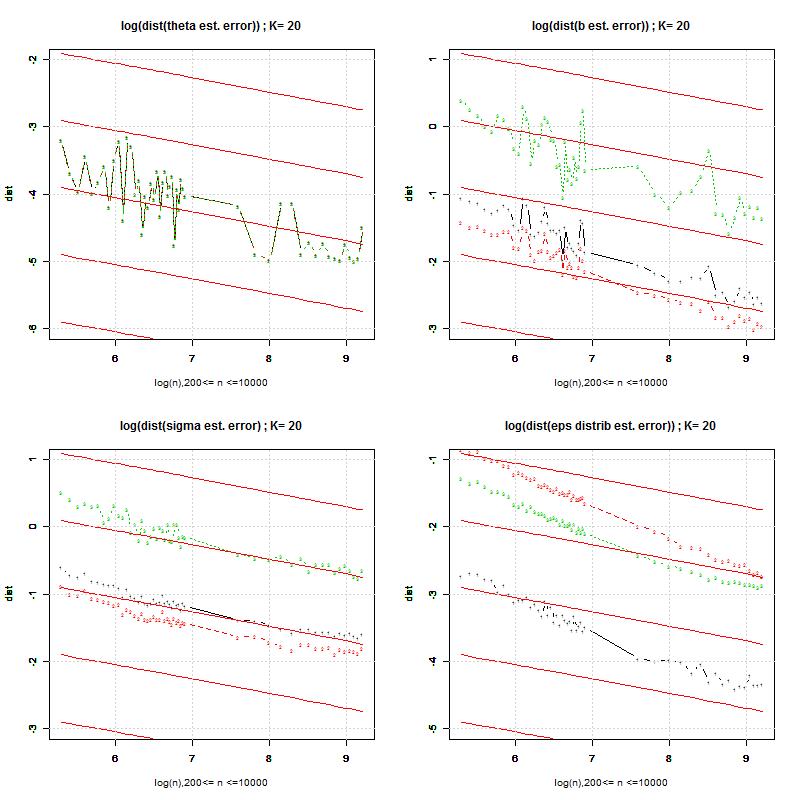}
	\caption{For model (\ref{-}), the graphics present, in log-log coordinates, the error in estimating the parameters $\theta$ (top-left), $b$ (top-right) and $\sigma$ (bottom-left) and the distance between the estimated distribution of the noise and the Gaussian distribution (bottom-right). The straight lines (almost straight, because of the term $\ln n$) show the curves  $c (\ln n/n)^{1/4}$ for several values of $c$.}
	\label{fig:speed2}
\end{figure}


\begin{figure} [htbp]
	\centering
		 \includegraphics[scale=.5]{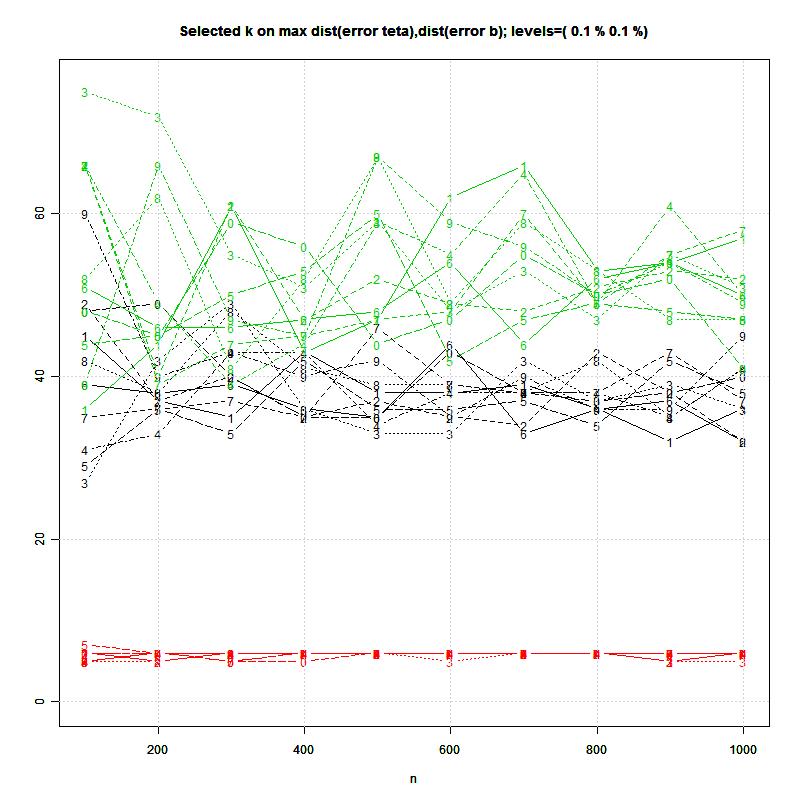}
	\caption{Value of $k(n)$ as a function of the sample size $n$ for the three models (\ref{+}) (the 5 upper lines), (\ref{-}) (the 5 lowest nearly constant lines) and   (\ref{4}) (the 5 intermediate lines).}
	\label{fig:choixk1}
\end{figure}

\newpage
\section{Appendix: Proofs}
\subsection{Preliminaries about the process $(X_n)$ and its covariances}\label{prel}

We consider the solution of (\ref{lina}) defined by the 
$MA_\infty$ expansion
\begin{equation}\label{sol}
X_{n}=\sum_{j\geq
  0}g_j\left(b(e_{n-j})+\sigma(e_{n-j})\varepsilon_{n-j}\right)\qquad n\in\Z
\end{equation}
where the geometrically vanishing sequence $(g_j)$ is defined by 
$$
\frac{1}{1-a_1z-\ldots -a_pz^p}=\sum_{j\geq 0}g_jz^j
$$

Since the sequence $s_n$ is $T$-periodic, and $\eta_n$ is i.i.d., the $T$-dimensional vector sequence
$Z_k={^t\left(X_{kT},\ldots,X_{(k+1)T-1}\right)}$ is a strictly stationary process, each
coordinate being the sum of $T$ linear scalar  processes based on $T$ independent white noises. In other words, the process $(X_n)$ is periodically correlated (see for example \cite{MR1724695} for a review on periodically correlated time series).

Hereafter, the stationarity of $Z_k$ is the key for proving convergence results via the law of large numbers.
\subsection{Proof of Lemma \ref{cov}}

The proof consists in separating the sequence $(\phi_k)$ into the $T$ stationary and ergodic subsequences $\{(\phi_k^{(l)})\vert l=0,\ldots,T-1\}$ defined in (\ref{fil}) and using the law of large numbers. Details are omitted. 

To check regularity of the limit $M$, consider the vector sequences $(\psi_k)$ and $(\psi_k^{(l)})$ built from 
$$
Y_n=\sum_{j\geq 0}g_j\sigma(e_{n-j})\varepsilon_{n-j}
$$
exactly as $(\phi_k)$ and $(\phi_k^{(l)})$ are built from $(X_n)$.  Similarly, consider the sequences $(\psi_k')$ and $(\psi_k'^{(l)})$ built from $
Y'_n=\sum_{j\geq 0}g_jb(e_{n-j}).
$
Denoting by $\Gamma'^{(l)}$ the covariance matrix of $(\psi_k'^{(l)})$, and noticing that the sequences $(\psi_k)$ and $(\psi_k')$ are orthogonal,  
$$
\Gamma^{(l)}=\Gamma'^{(l)}+\E\left(\psi_0^{(l)}{^t\psi_0^{(l)}}\right)
$$
Hence, if $M$ is singular, the same holds for $\sum_{l=0}^{T-1}\E\left( \psi_0^{(l)}{^t\psi_0^{(l)}} \right)$. This in turn implies that there exist $(c_1,\ldots,c_p)$ such that, for every $k$,
\begin{equation}\label{S}
c_1 S_k+\cdots+c_p S_{k-p+1}=_{as} 0
\end{equation}
where $S_k=Y_k+\ldots+Y_{k-T+1}$ is the sum of the $Y$'s over a period of the input $e_n$. Now, it is clear that $S_k$ is a stationary ARMA process having the representation
$$
S_k=a_1S_{k-1}+\ldots+a_pS_{k-p+1}+\sum_{j=0}^{T-1}\sigma(e_{k-j})\varepsilon_{k-j},
$$
where, from (\ref{varnoise}), the variance of the noise is not zero.
This contradicts (\ref{S}).

\subsection{Proof of Theorem \ref{resum0}}

Most proofs below are classical in the field of kernel functional
estimation. This is why some details are omitted. The reader can refer to \cite{Bosq1}, \cite{Fan2} or \cite{Fer} for complete devlopments.
\subsubsection{Convergence of $R_n^{(1)}$ and $R_n^{(2)}$}
\begin{Lem}\label{R12}
Under the assumptions $\mathcal H_{1,\ldots,8}$, and if the smoothing parameter $h_n$ is such that  $h_n\sim n^{\beta_1}(\ln n)^{\beta_2}$ with some $\beta_1<0$
we have
$$
R_n^{(1)}=O_{a.s.}\left(\sqrt \frac{\ln n}{nh_n}\right)+O(h_n^{\gamma})
$$
\end{Lem} 
\begin{proof}
Given the convergence of $\Sigma_n/n$ to a regular matrix, it is enough to prove the wanted result for
\begin{equation}\label{1+2}
\frac{1}{n}\sum_{l=p+1}^{n}\phi_l
\frac{\sum_{j=p+1}^{n}\left(b(e_j)-b(e_l)+\sigma(e_j)\varepsilon_j\right)K_n\left(e_l-e_j\right)}{\sum_{j=p+1}^{n}K_n\left(e_l-e_j\right)}.
\end{equation}
We prove the uniform convergence
\begin{equation}\label{unifnum}
\sup_e\Big\vert\frac{\sum_{j=p+1}^{n}\left(b(e_j)-b(e)+\sigma(e_j)\varepsilon_j\right)K_n\left(e-e_j\right)}{\sum_{j=p+1}^{n}K_n\left(e-e_j\right)}\Big\vert=
O_{a.s.}\left(\sqrt \frac{\ln n}{nh_n}\right)+O(h_n^{\gamma}).
\end{equation}

The result will then follow from the fact that, thanks to the law of large numbers applied to each subsequence $(\phi_k^{(l)})_k,\; (k=0,\ldots,T-1)$, the arithmetic mean  $n^{-1}\sum_{p+1}^{n}\phi_k$ almost surely converges.
\end{proof}
In order to prove (\ref{unifnum}) we only consider
\begin{equation}\label{b}
\frac{\sum_{j=p+1}^{n}\left(b(e_j)-b(e)\right)K_n\left(e-e_j\right)}{\sum_{j=p+1}^{n}K_n\left(e-e_j\right)}=\frac{\frac{1}{nh_n}\sum_{j=p+1}^{n}\left(b(e_j)-b(e)\right)K_n\left(e-e_j\right)}{\frac{1}{nh_n}\sum_{j=p+1}^{n}K_n\left(e-e_j\right)}.
\end{equation}
The treatment of the other part in (\ref{1+2}) is simpler since 
$\E(\sigma(e_j)\varepsilon_jK_n\left(e-e_j\right))=0$ for every $j$. 

$\bullet$ Consider first the numerator of (\ref{b}) conveniently split in two parts: a variance term and a biais trem

\begin{eqnarray*}\label{var}
N_1(e)&=&\frac{1}{nh_n}\sum_{j=p+1}^{n}(b(e_j)-b(e))K_n\left(e-e_j\right)-\E\left[(
(b(e_j)-b(e))K_n\left(e-e_j\right)\right]
\end{eqnarray*}
and
\begin{equation*}\label{biais}
N_2(e)=\frac{1}{nh_n}\sum_{j=p+1}^{n}\E\left[((b(e_j)-b(e))K_n\left(e-e_j\right)\right].\nonumber
\end{equation*}
For  the so-called variance term $N_1(e)$,
the basic tool is the  exponential inequality 

\begin{equation}\label{Hoe}
P\left(\left\vert\frac{\sum_{j=1}^nU_j}{n}\right\vert>\epsilon\right)\leq
2e^{-\frac{n\epsilon^2}{4\delta^2}}\quad \forall \epsilon\in]0,3\delta^2/d[,
\end{equation}
which holds for every set $(U_1,\ldots,U_n)$ of independent zero-mean variables such
that $\vert U_j\vert\leq d$ and $\E(U_j^2)\leq \delta^2$  $(j=1,\ldots, n)$. This inequality is easily deduced from Bernstein's one as noticed in \cite{Hoef}, page 17.

Looking at the independent sequence
$$
U_j=\frac{1}{h_n}\left((b(e_j)-b(e))K_n(e_j-e)-\E((b(e_j)-b(e))K_n(e_j-e))\right), 
$$
firstly, since  $b$ and  $K$ are bounded, it is clear that
$$
\vert U_j\vert\leq \frac{c}{h_n},
$$
and secondly
\begin{eqnarray*}
\E(U_j^2)&\leq& \frac{1}{h_n^2}\int(b(u)-b(e))^2K^2\left(\frac{u-e}{h_n}\right)f(u)du\\
&=&\frac{1}{h_n}\int(b(vh_n+e)-b(e))^2K^2(v) f(vh_n+e) dv\leq \frac{c}{h_n}.
\end{eqnarray*}
Applying inequality (\ref{Hoe}) with $d=\delta^2=c/h_n$  we obtain
$$
P\left(\vert N_1(e)\vert>\epsilon\right)\leq 2e^{-\frac{nh_n\epsilon^2}{4c}}\quad 0<\varepsilon<1
$$
and then
$$
P\left(\vert N_1(e)\vert>\epsilon_0\sqrt\frac{\ln n}{nh_n}\right)\leq
2e^{-\frac{\varepsilon_0^2\ln n}{4c}}.
$$
A suitable choice of $\epsilon_0$ yields summability of the r.h.s. and finally, by Borel Cantelli Lemma
$$
N_1(e)=O_{a.s.}\left(\sqrt\frac{\ln n}{nh_n}\right).
$$

We now turn to the biais term $N_2(e)$. From (\ref{Hol}), 
\begin{eqnarray*}
N_2(e)&=&\frac{1}{h_n}\int(b(u)-b(e))K\left(\frac{u-e}{h_n}\right)f(u)du\\
&=&\int (b(vh_n+e)-b(e))K(v)f(vh_n+e)dv=O(h_n^\gamma).
\end{eqnarray*}
We have thus proved that 
\begin{eqnarray*}
N_1(e)+N_2(e)&=&\frac{1}{nh_n}\sum_{j=p+1}^{n}\left(b(e_j)-b(e)\right)K_n\left(e-e_j\right)\\
&=&O_{a.s.}\left(\sqrt\frac{\ln n}{nh_n}\right)+O(h_n^\gamma).
\end{eqnarray*}
The same rate for $\sup_e (\vert N_1(e)+N_2(e)\vert)$ is obtained by covering the domain of $e$ by well chosen intervals and using Lipschitz property of the kernel. See \cite{Bosq1} and \cite{Fer} among others for the details.

$\bullet$ A similar treatment leads to
\begin{equation}\label{den}
\sup_{e\in \mathcal E}\Bigg\vert
\frac{\sum_{j=p+1}^{n}K_n\left(e-e_j\right)}{nh_n}-\frac{
\sum_{l=0}^{T-1}f(e-s_l)}{T}
\Bigg\vert\tops 0
\end{equation}
This, together with the fact that $\inf _{e\in \mathcal E} f(e)>0$, leads to
\begin{eqnarray*}
&&\sup_l\Bigg\vert\frac{\sum_{j=p+1}^{n}\left(b(e_j)-b(e_l)\right)K_n\left(e_l-e_j\right)}{\sum_{j=p+1}^{n}K_n\left(e_l-e_j\right)}\Bigg\vert\\
&\leq&
\sup_{e\in \mathcal E}\Bigg\vert\frac{\sum_{j=p+1}^{n}\left(b(e_j)-b(e)\right)K_n\left(e-e_j\right)}{\sum_{j=p+1}^{n}K_n\left(e-e_j\right)}\Bigg\vert
=O_{a.s.}\left(\sqrt\frac{\ln n}{nh_n}\right)+O(h_n^\gamma)
\end{eqnarray*}
and the proof of (\ref{unifnum}) is over.

Let us now consider the convergence of $R_n^{(2)}$:
\begin{Lem}\label{R3}
Under the assumptions $\mathcal H_{1,2,3,4}$, 
$$
R_n^{(2)}=o_{a.s.}\left({n^{\gamma}}\right)\quad \forall \gamma>-1/2
$$
\end{Lem}
\begin{proof}
The vector sequence $\phi_k\sigma(e_k)\varepsilon_k$ is a martingale difference sequence since $\E(\varepsilon_k)=0$ and since $\phi_k e_k$ and  $\varepsilon_k$ are independent. Moreover,
$$
\E(\normdeux{\phi_k\sigma(e_k)\varepsilon_k}^2)=\sigma^2\E\left(b(e_k)^2\right)\E(\normdeux{\phi_k}^2).
$$
where $\E(\normdeux{\phi_k}^2)$ and $\E\left(b(e_k)^2\right)$ are
periodic. Hence, for every $\beta>1/2$
$$
\sum\frac{\E(\normdeux
  {\phi_k\sigma(e_k)\varepsilon_k}^2)}{k^{2\beta}}<\infty, 
$$
implying, from theorem 3.3.1 of \cite{Stout},
$$
n^{-\beta}\sum_{p+1}^n\phi_k\sigma(e_k)\varepsilon_k\tops 0.
$$ 
Finally, the convergence of $\Sigma_n/n$ leads to the conclusion. 
\end{proof}
\subsubsection{Convergence of the coefficient $A_n$}
We prove the convergence of $A_n$, the matrix coefficient of $\tilde \theta_n ^{(k-1)}$ in (\ref{errtheta}). 
\begin{Lem}\label{AAPE}
Under the assumptions $\mathcal H_{1,\ldots,8}$,\\
(i) As $n\to\infty$, 
\begin{eqnarray*}
A_n&=&\Sigma_n^{-1}\left(\sum_{l=p+1}^n\phi_l
\frac{\sum_{j=p+1}^n{^t\phi_j}K_n(e_l-e_j)}{\sum_{j=p+1}^n K_n(e_l-e_j)}\right)\\
&\tops& M^{-1}\sum_{l,j=0}^{T-1}\mu^{(l)}{^t\mu^{(j)}}\int\frac{f(u-s_j)f(u-s_l)}{\sum_{i=0}^{T-1}f(u-s_i)}du =:A
\end{eqnarray*}
where $M$ is defined in Lemma \ref{cov}.\\
(ii) Moreover $\norm{A_n-A}= 0_{as}\left(\sqrt \frac{\ln n}{nh_n}\right)+O(h_n^{\gamma})$. 
\end{Lem}
\begin{proof}
We consider first 
$$
R_n(e):=\frac{\sum_{j=p+1}^n{^t\phi_j}K_n\left(e-e_j\right)}{\sum_{j=p+1}^{n}K_n\left(e-e_j\right)}=\frac{\frac{\sum_{j=p+1}^n{^t\phi_j}K_n\left(e-e_j\right)}{nh_n}}{\frac{\sum_{j=p+1}^{n}K_n\left(e-e_j\right)}{nh_n}}.
$$
The denominator has been already treated in the proof of Lemma \ref{R12} (see
(\ref{den})), so we focus on the numerator 
and successively show that
\begin{equation}\label{num}
\sup_{e\in \mathcal E}\Bigg\vert \frac{\sum_{j=p+1}^n{^t\phi_j}K_n\left(e-e_j\right)-\E(^t\phi_j K_n\left(e-e_j\right))}{nh_n}\Bigg\vert=O_{a.s.}\left(\sqrt\frac{\ln n}{nh_n}\right),
\end{equation}
then, with $\phi_k^{(l)}$ defined in (\ref{fil}),

\begin{equation}\label{numesp}
\sup_{e\in \mathcal E}\Bigg\vert \frac{\sum_{j=p+1}^{n}\E(^t\phi_j K_n\left(e-e_j\right))}{nh_n}-\frac{\sum_{l=0}^{T-1}{^t\mu^{(l)}} f(e-s_l)}{T}\Bigg\vert=O(h_n^\gamma)
\end{equation}

The proof of (\ref{numesp}) uses $\int K(e)de=1$. The details are omitted. 
The proof of (\ref{num}) follows the lines of the proof of (\ref{unifnum}),
the difference coming from the fact that the $(\phi_j
K_n(e_{j}-e))_k$ are not independent. In fact, they are weakly dependent in so far as, conditionally to the exogeneous sequence, they are mixing.

\begin{Lem}\label{mix0}~

(i) For every $e\in\mathcal E$ and every $h$, the vector sequence $(\phi_j
K(\frac{e_{j}-e}{h}))_j$ is, conditionally to the sequence $(e_j)_j=:\overline E$, geometrically $\alpha$-mixing. 

(ii) This property holds uniformly with respect to $\overline E$: there exists a constant $C$ and $\alpha\in ]0,1[$ such that, $\alpha^{\overline E}(n)$ being the conditional mixing sequence,
$$
\alpha^{\overline E}(n)\leq C \alpha^n\qquad\forall n.
$$
\end{Lem}
\begin{proof} Consider for example the first coordinate $K(\frac{e_{j}-e}{h})X_{j-1}$ of the vector sequence. Conditionally to $\overline E$, the sequence $K(\frac{e_{j}-e}{h})$ is deterministic, and it is enough to consider the sequence $X_{j}$ which has the same conditional mixing coefficients as $K(\frac{e_{j}-e}{h})X_{j-1}$. From (\ref{sol})
\begin{equation*}
X_{n}=\sum_{j\geq 0}g_j\left(b(e_{n-j})+\sigma(e_{n-j})\varepsilon_{n-j}\right)
\end{equation*}
is a linear time series based on the bounded noise 
$b(e_{j})+\sigma(e_{j})\varepsilon_{j}$, where $b(e_{j})$ and $\sigma(e_{j})$ are deterministic trend and variance, while $\varepsilon_j$ is i.i.d. 
Let $h_j(u)$ be the conditional density of the noise. We obtain, since
$g$ is $C_1$ and $\inf_{e\in\mathcal E}\sigma(e)>0$, 
\begin{eqnarray*}
&&\int\vert h_j(u+x)-h_j(u)\vert du\\
&\leq&\int\frac{1}{\sigma(e_j)}\left\vert g\left(\frac{u+x-b(e_j)}{\sigma(e_j)}\right)-g\left(\frac{u-b(e_j)}{\sigma(e_j)}\right)\right\vert\
f(v)dv\\
&\leq& \frac{\normsup g'}{\inf_{e\in\mathcal E}\sigma^2(e)}\vert x \vert.
\end{eqnarray*}
Now, the sequence $(X_{j})_j$ is bounded and, for every $j$, $\vert g_j\vert\leq C\beta ^j$ for a certain $\beta\in ]0,1[$. Hence the theorem in \cite{Gor} applies, with any $0<\delta<1$: the sequence $K(\frac{e_{j}-e}{h})X_{j-1}$ is conditionally $\alpha$-mixing, with mixing coefficients satisfying 
$$
\alpha^{\overline E}(n)\leq C \left(\beta^{\frac{\delta}{1+\delta}}\right)^n=:C\alpha_1^n\qquad\forall n
$$
where the constant $C$ does not depend on $\overline E$.
\end{proof}
The reader is referred to \cite{Dou} for definitions and properties of mixing
sequences. Hereafter we need to replace inequality (\ref{Hoe}) by the following one, a direct consequence of theorem 6.2 in
\cite{Rio}: 

\begin{Lem}
Let $(V_j)$ be a strong mixing sequence of centered 
random variables such that
$$
\alpha(n)\leq c\alpha^n,\;\forall n \quad\hbox{and}\quad \vert V_j\vert\leq M,\;\forall j
$$ Denote $s_n^2=\sum_{1\leq j,k\leq n}\vert
\cov(V_j,V_k)\vert$. For any $r>1$ and $\lambda>0$,
\begin{equation}\label{expmix}
P\left(\left\vert\sum_{j=1}^n V_j\right\vert>4\lambda\right)\leq 4\left(1+\frac{\lambda^2}{rs_n^2}\right)^{-r/2}+\frac{4Mcn}{\lambda}\alpha^{\frac{\lambda}{Mr}}.
\end{equation}
\end{Lem}

This inequality applies, conditionally to $\overline E$, to
$$
V_j=^t\phi_j K_n\left(e-e_j\right)-\E(^t\phi_j K_n\left(e-e_j\right)), \quad j\geq p+1.
$$
For this sequence $V_j$, the conditional variance $s_n^2$ satisfies
\begin{equation}\label{somcov}
s_n^2=O(nh_n)
\end{equation}
where the $O$ is uniform with respect to $\overline E$.
Indeed,
\begin{displaymath}
\left\{\begin{array}{lll}
\var^{\overline E}(V_j)\leq c h_n&&\\
\vert \cov^{\overline E}(V_j,V_l)\vert\leq h_n^2&\hbox{if}\quad \vert j-l\vert\leq \delta_n\\
\vert\cov^{\overline E}(V_j,V_l)\vert\leq C\alpha_1^{\vert j-l\vert}&\hbox{if}\quad \vert j-l\vert> \delta_n
\end{array}\right.
\end{displaymath}
For the last bound, the reader can refer to \cite{Dou}. The two first ones are
directly obtained.
Taking $\delta_n =1/(h_n\ln n)$ easily leads to (\ref{somcov}).

Now, with $M:=\normsup K esssup_j \vert X_j \vert$,  (\ref{expmix}) leads to
\begin{eqnarray*}
P^{\overline E}\left(\Big\vert \sum_{j=p+1}^n{^t\phi_j}K_n\left(e-e_j\right)-\E(^t\phi_j K_n\left(e-e_j\right))\Big\vert>4\lambda\right)
&\leq&4\left(1+\frac{c\lambda^2}{r n h_n}\right)^{-r/2}\\&+&\frac{4MCn}{\lambda}\alpha_1^{\frac{\lambda}{Mr}},
\end{eqnarray*}
and then, if $\ln n=o(r_n)$
\begin{eqnarray*}
&&P^{\overline E}\left(\Bigg\vert \frac{\sum_{j=p+1}^n{^t\phi_j}K_n\left(e-e_j\right)-\E(^t\phi_j K_n\left(e-e_j\right))}{nh_n}\Bigg\vert>\lambda_0\sqrt\frac{\ln n}{nh_n}\right)\\
&\leq& 4\left(1+\frac{c\lambda_0^2\ln n}{16 r_n}\right)^{-r_n/2}+\frac{16MCn}{\lambda_0\sqrt{nh_n\ln n}}\alpha_1^{\frac{\lambda_0\sqrt{nh_n\ln n}}{4Mr_n}}\\
&\leq& 4e^{-\frac{c\lambda_0^2\ln n}{32}}+\frac{C_1}{\lambda_0}\sqrt{\frac{n}{h_n\ln n}}\alpha_1^{\frac{\lambda_0\sqrt{nh_n\ln n}}{4Mr_n}}.
\end{eqnarray*}
Now,  
if $h_n\sim n^{\beta_1}\ln n^{\beta_2}$ with $\beta_1>-1$, $r_n=(\ln n)^\beta$  we get, for $n$ large enough,
\begin{eqnarray}
&&P^{\overline E}\left(\Bigg\vert \frac{\sum_{j=p+1}^n{^t\phi_j}K_n\left(e-e_j\right)-\E(^t\phi_j K_n\left(e-e_j\right))}{nh_n}\Bigg\vert>\lambda_0\sqrt\frac{\ln n}{nh_n}\right)\nonumber\\
&\leq& 4n^{-c\lambda_0^2}+C_2\frac{n^{\frac{1-\beta_1}{2}}}{(\ln n)^{\frac{1+\beta_2}{2}}}\alpha_1^{\frac{\lambda_0n^{\frac{1+\beta_1}{2}}(\ln n)^{\frac{1+\beta_2}{2}-\beta}}{4M}}\nonumber\\
&\leq& 4n^{-c\lambda_0^2}+C_2n^{\frac{1-\beta_1}{2}+\frac{\lambda_0\ln \alpha_1}{4M}}\label{majcond}.
\end{eqnarray}
Now, the constants in (\ref{majcond}) do not depend on $\overline E$, implying that
\begin{eqnarray*}
&&P\left(\Bigg\vert \frac{\sum_{j=p+1}^n{^t\phi_j}K_n\left(e-e_j\right)-\E(^t\phi_j K_n\left(e-e_j\right))}{nh_n}\Bigg\vert>\lambda_0\sqrt\frac{\ln n}{nh_n}\right)\\
&\leq& 4n^{-c\lambda_0^2}+C_2n^{\frac{1-\beta_1}{2}+\frac{\lambda_0\ln \alpha_1}{4M}}.
\end{eqnarray*}
and it is easy to select  $\lambda_0$ for the r.h.s. to be
the general term of a convergent series. 

So we have proved that, for fixed  $e$, 
\begin{equation*}
\frac{\sum_{j=p+1}^n{^t\phi_j}K_n\left(e-e_j\right)-\E(^t\phi_j K_n\left(e-e_j\right))}{nh_n}=O_{a.s.}\left(\sqrt\frac{\ln n}{nh_n}\right).
\end{equation*}
The same speed is obtained for the $\sup$-norm.

\noindent From (\ref{num}), (\ref{numesp}) and  (\ref{den}) it follows that, with 
$$
\tilde R(e):=\frac{\sum_{j=0}^{T-1}{^t\mu_j}f(e-s_j)}{\sum_{j=0}^{T-1}f(e-s_j)}
$$
\begin{equation}\label{R}
\sup_e\Big\vert R_n(e)-\tilde R(e)\Big\vert=O_{a.s.}\left(\sqrt\frac{\ln n}{nh_n}\right)+O_{a.s.}(h_n^{\gamma})
\end{equation}
implying in turn 
\begin{eqnarray}
A_n&=&n\Sigma_n^{-1}\frac{1}{n}\sum_{l=p+1}^n\phi_l R_n(e_l)\nonumber\\
&=&n\Sigma_n^{-1}\frac{1}{n}\sum_{l=p+1}^n\phi_l(R_n(e_l)-\tilde R(e_l))+n\Sigma_n^{-1}\frac{1}{n}\sum_{l=p+1}^n\phi_l\tilde R(e_l)\nonumber\\
&=&O_{a.s.}\left(\sqrt\frac{\ln n}{nh_n}\right)+O_{a.s.}(h_n^{\gamma})+n\Sigma_n^{-1}\frac{1}{n}\sum_{l=p+1}^n\phi_l\tilde R(e_l).\label{last}
\end{eqnarray}
In (\ref{last}), the last sum is separated into $T$ sums
$$
\frac{1}{n}\sum_{l=p+1}^n\phi_l\tilde R(e_l)=\sum_{l=0}^{T-1}\frac{1}{n}\sum_{kT+l\leq n}^n\phi^{(l)}_k\tilde R(s_l+\eta_{kT+l}),
$$
which almost surely converges to 
\begin{eqnarray*}
\frac{1}{T}\sum_{l=0}^{T-1}\E\left(\phi^{(l)}_0\right)E(\tilde R(s_l+\eta_0))&=&\frac{1}{T}\sum_{l=0}^{T-1}\mu^{(l)}E(\tilde R(s_l+\eta_0))\\
&=&\frac{1}{T}\sum_{l,j=0}^{T-1}\mu^{(l)}{^t\mu^{(j)}}\int\frac{f(v-s_j)f(v-s_l)}{\sum_{i=0}^{T-1}f(v-s_i)}dv
\end{eqnarray*}
Moreover, this convergence rate, being the rate in the law of large numbers for i.i.d sequences, is faster than the first two terms in (\ref{last}). This, together with (\ref{last}) and the almost sure convergence of $n\Sigma_n^{-1}$, leads to the desired result. Lemma \ref{AAPE} is proved.

\end{proof}
Lemma \ref{AAPE}, together with Lemma \ref{supp} below, shows that the passage
(\ref{errtheta}) from step $k-1$ to step $k$ is a fixed point iteration, at least for $n$ large enough. 
\begin{Lem}\label{supp}
There exists $k_0\geq 1$ such that
\begin{equation}\label{kper}
\sup_v\frac{{\normdeux{A^{k_0} v}}}{\normdeux{v}}<1.
\end{equation}
Moreover, $k_0=1$ when $p=1$.
\end{Lem}
\begin{proof}
For the sake of simplicity, we take $T=2$. The general case only brings more complicated formulas.
Denoting 
$$
S=\Gamma^{(1)}+\Gamma^{(2)},
$$
and 
\begin{equation*}
\alpha_{jl}=\int\frac{f(v-s_j)f(v-s_l)}{\sum_{i=0}^{1}f(v-s_i)}dv,
\end{equation*}
\begin{eqnarray}\label{0}
A=\left[S+\mu_0^t\mu_0+\mu_1^t\mu_1\right]^{-1}\left(\alpha_{00}\mu_0^t\mu_0+\alpha_{11}\mu_1^t\mu_1+\alpha_{01}(\mu_0^t\mu_1+\mu_1^t\mu_0)\right).
\end{eqnarray}
We then apply a popular matrix inversion formula:
\begin{eqnarray*}
\left[S+\mu_0^t\mu_0+\mu_1^t\mu_1\right]^{-1}\mu_1&=&\frac{[S+\mu_0^t\mu_0]^{-1}\mu_1}{1+^t\mu_1[S+\mu_0^t\mu_0]^{-1}\mu_1}=\frac{S_{1}^{-1}\mu_1}{1+^t\mu_1M_{1}^{-1}\mu_1}\\
\left[S+\mu_0^t\mu_0+\mu_1^t\mu_1\right]^{-1}\mu_0&=&\frac{[S+\mu_1^t\mu_1]^{-1}\mu_0}{1+^t\mu_0[S+\mu_1^t\mu_0]^{-1}\mu_1}=\frac{S_{0}^{-1}\mu_0}{1+^t\mu_0S_{0}^{-1}\mu_0}
\end{eqnarray*}
where
$$
S_1=S+\mu_0^t\mu_0\quad\hbox{et}\quad S_0=S+\mu_1^t\mu_1.
$$
This leads to
$$
A=\frac{S_{0}^{-1}\mu_0}{1+{^t\mu_0}S_{0}^{-1}\mu_0}\left(\alpha_{00}{^t\mu_0}+\alpha_{01}{^t\mu_1}\right)+\frac{S_{1}^{-1}\mu_1}{1+{^t\mu_1}S_{1}^{-1}\mu_1}
\left(\alpha_{11}{^t\mu_1}+\alpha_{01}{^t\mu_0}\right)
$$
and finally to 
\begin{eqnarray}
A&=&\alpha_{00}\frac{S_{0}^{-1}\mu_0^t\mu_0}{1+{^t\mu_0}S_{0}^{-1}\mu_0}+\alpha_{11}\frac{S_{1}^{-1}\mu_1^t\mu_1}{1+{^t\mu_1}S_{1}^{-1}\mu_1}\nonumber\\
&+&\alpha_{01}\left(\frac{S_{0}^{-1}\mu_0^t\mu_1}{1+{^t\mu_0}S_{0}^{-1}\mu_0}+\frac{S_{1}^{-1}\mu_1^t\mu_0}{1+{^t\mu_1}S_{1}^{-1}\mu_1}
\right)\nonumber\\
&=&\alpha_{00}S_{00}+\alpha_{11}S_{11}+\alpha_{01}(S_{01}+S_{10})\label{A}
\end{eqnarray}
where the last line defines the $S_{jl}$'s.

It is easily checked that
\begin{eqnarray*}
S_{jj}^2&=&\frac{{^t\mu_j}S_{j}^{-1}\mu_j}{1+{^t\mu_j}S_{j}^{-1}\mu_j}S_{jj}=\beta_{jj}S_{jj},\quad j=1,2\\
S_{jk}^2&=&\frac{{^t\mu_k}S_{j}^{-1}\mu_j}{1+{^t\mu_j}S_{j}^{-1}\mu_j}S_{jk}=\beta_{jk}S_{jk}\quad j\neq k
\end{eqnarray*}
Clearly, $0\leq \beta_{jj}<1$. Moreover,
$$
\vert \beta_{jk}\vert\leq\frac{\sqrt{{^t\mu_j}S_{j}^{-1}\mu_j}}{1+{^t\mu_j}S_{j}^{-1}\mu_j}\sqrt{{^t\mu_k}S_{j}^{-1}\mu_k}<1
$$
because the first factor is less than $1/2$, and 
\begin{eqnarray*}
{^t\mu_k}S_{j}^{-1}\mu_k&=&{^t\mu_k}[S+\mu_k^t\mu_k]^{-1}\mu_k={^t\mu_k}\left(S^{-1}-\frac{S^{-1}\mu_k^t\mu_kS^{-1}}{1+^t\mu_kS^{-1}\mu_k}\right)\mu_k\\
&=&\frac{^t\mu_kM^{-1}\mu_k}{1+^t\mu_kM^{-1}\mu_k}<1.
\end{eqnarray*}
 
As $\alpha_{jl}\in [0,1]$ for every $j,l$, it results that 
$$
A^2=\alpha^{(2)}_{00}S_{00}+\alpha^{(2)}_{11}S_{11}+\alpha^{(2)}_{01}(S_{01}+S_{10})
$$
where for every $j,l$, $\vert\alpha^{(2)}_{j,l}\vert\leq \beta_{jl}\alpha_{j,l}$, whence
$$
A^k=\alpha^{(k)}_{00}M_{00}+\alpha^{(k)}_{11}M_{11}+\alpha^{(k)}_{01}(M_{01}+M_{10})
$$
where for every $j,l$, $\vert\alpha^{(k)}_{j,l}\vert\leq (\beta_{jl})^{k-1}\alpha_{j,l}$. Lemma \ref{supp} is proved.
\end{proof}
It remains to prove (\ref{variance}), the rate of convergence of the error on the standard deviation.
The estimation error $\tilde \sigma_{n,k}(e)$ is
\begin{eqnarray}
\tilde \sigma_{n,k}(e)&=&\hat \sigma_{n,k}(e)-\sigma^2(e)\nonumber\\
&=&\frac{\sum_{l=p+1}^{n-1}\left(\left(X_l-^t\phi_l\hat\theta^{(k)}_n-\hat b_n^{(k-1)}(e_{l})\right)^2-\sigma^2(e)\right)K_n\left(e-e_l\right)}{\sum_{l=p+1}^{n-1}K_n\left(e-e_l\right)}\nonumber\\
&=&\frac{\sum_{l=p+1}^{n-1}\left(\sigma^2(e_l)\varepsilon_l^2-\sigma^2(e)\right)K_n\left(e-e_l\right)}{\sum_{l=p+1}^{n-1}K_n\left(e-e_l\right)}+R_{n,k}(e)\label{sigma1}
\end{eqnarray}
where, from the first part of the theorem, 
$$
R_{n,k}(e)=O_{a.s.} \left(\sqrt \frac{\ln
 n}{nh_n}\right)+O_{a.s.}(h_n^\gamma)+O_{a.s.}(\beta^{k}).
$$
Now, since the variables $\sigma^2(e_l)\varepsilon_l^2-\sigma^2(e)$ are independent and centered,  the first term in (\ref{sigma1}) can be treated exactly as was (\ref{b}), leading to 
$$
\frac{\sum_{l=p+1}^{n-1}\left(\sigma^2(e_l)\varepsilon_l^2-\sigma^2(e)\right)K_n\left(e-e_l\right)}{\sum_{l=p+1}^{n-1}K_n\left(e-e_l\right)}=O_{a.s.} \left(\sqrt \frac{\ln
 n}{nh_n}\right)+O_{a.s.}(h_n^{\gamma'})
$$
and the proof of (\ref{variance}) is completed.

\end{document}